\documentclass[12pt]{article}
\usepackage{epsf}

\usepackage{amsthm}
\usepackage{amssymb}
\usepackage{amsmath}

\usepackage[top=1.5cm,bottom=1.5cm,left=2.5cm,right=2.5cm,includehead,includefoot
           ]{geometry}

\makeatletter

\numberwithin{equation}{section}
\newtheorem{theorem}{Theorem}[section]
\newtheorem{lemma}[theorem]{Lemma}
\newtheorem{corollary}[theorem]{Corollary}

\title{Derivations, local and 2-local derivations on some algebras of operators on Hilbert C*-modules}
\author{\begin{tabular}{c} Jun He, Jiankui Li\footnote{Corresponding author.
E-mail address: jiankuili@yahoo.com},  and Danjun Zhao
\\{\small\it Department of Mathematics, East China University of
Science and Technology}\\
{\small\it Shanghai 200237, China}
\end{tabular}}
\date{}
\begin{document}
\maketitle \abstract
For a commutative C*-algebra $\mathcal A$ with unit $e$ and a Hilbert~$\mathcal A$-module $\mathcal M$,
denote by End$_{\mathcal A}(\mathcal M)$ the algebra of all bounded $\mathcal A$-linear mappings on $\mathcal M$,
and by End$^*_{\mathcal A}(\mathcal M)$ the algebra of all adjointable mappings on $\mathcal M$.
We prove that if $\mathcal M$ is full, then each derivation on End$_{\mathcal A}(\mathcal M)$ is $\mathcal A$-linear,
continuous, and inner, and each 2-local derivation on End$_{\mathcal A}(\mathcal M)$
or End$^{*}_{\mathcal A}(\mathcal M)$ is a derivation. If there exist $x_0$ in $\mathcal M$
and $f_0$ in $\mathcal M^{'}$, such that $f_0(x_0)=e$, where $\mathcal M^{'}$ denotes the set of all bounded
$\mathcal A$-linear mappings from $\mathcal M$ to $\mathcal A$, then each $\mathcal A$-linear local derivation on
End$_{\mathcal A}(\mathcal M)$ is a derivation.

\

\

\
\section{Introduction and preliminaries}\

The structure of derivations on operator algebras is an important part of the theory of operator algebras.

Let $\mathcal{A}$ be an algebra and $\mathcal{M}$ be an $\mathcal{A}$-bimodule.
Recall that a \emph{derivation} is a linear mapping $d$ from $\mathcal A$ into $\mathcal{M}$
such that $d(xy) = d(x)y + xd(y)$, for all $x, y$ in $\mathcal A$. For each $m$ in $\mathcal M$,
one can define a derivation $D_m$ by $D_m(x) =mx - xm$, for all $x$ in $\mathcal A$.
Such derivations are called \emph{inner derivations}.

It is a classical problem to identify those algebras on which all derivations are inner derivations. Several authors
investigate this topic. The following two results are classical. S. Sakai \cite{sakai} proves that all derivations from a W*-algebra into itself are inner derivations.
E. Christensen \cite{Christensen} proves that all derivations from a nest algebra into itself are inner derivations.

In 1990, R. Kadison \cite{Kadison} and D. Larson, A. Sourour \cite{Larson} independently introduce
the concept of local derivation in the following sense:
a linear mapping $\delta$ from $\mathcal{A}$ into $\mathcal{M}$ such that
for every $a\in$ $\mathcal{A}$, there exists a derivation $d_a :$
$\mathcal{A}$ $\rightarrow$ $\mathcal{M}$, depending on $a$,
satisfying $\delta(a)=d_a(a)$. In \cite{Kadison}, R. Kadison proves that each continuous local
derivation from a von Neumann algebra into its dual Banach
module is a derivation. In \cite{Larson}, D. Larson and A. Sourour prove that each local derivation
from $B(\mathcal X)$ into itself is a derivation,
where $\mathcal X$ is a Banach space.
B. Jonson \cite{Johnson} proves that each local derivation from a C*-algebra
into its Banach bimodule is a derivation.
Z. Pan and the second author of this paper \cite{Li} prove that each local derivation from the algebra $\mathcal M \bigcap alg\mathcal L$
into $B(\mathcal H)$ is a derivation, where $\mathcal H$ is a Hilbert space, $\mathcal M$ is a von Neumann algebra acting on $\mathcal H$, 
and $\mathcal L$ is a commutative subspace lattice in $\mathcal M$.
For more information about this topic, we refer to \cite{Ayupov2, Crist, Hadwin}.

In 1997, P. $\check{S}$emrl \cite{Semrl} introduces the concept of 2-local derivations. Recall
that a mapping $\delta$ : $\mathcal{A}$ $\rightarrow$ $\mathcal{M}$ (not necessarily linear)
is called a \emph{2-local derivation} if for each $a,b\in$ $\mathcal{A}$, there exists a derivation
$d_{a,b}$ : $\mathcal{A}$ $\rightarrow$ $\mathcal{M}$ such that  $\delta(a)=d_{a,b}(a)$ and
$\delta(b)=d_{a,b}(b)$. Moreover, the author proves that every 2-local derivation
on $B(\mathcal H)$ is a derivation for a separable Hilbert space $\mathcal H$. J. Zhang and H. Li \cite{Zhang} extend the above
result for arbitrary symmetric digraph matrix algebras and construct an example of 2-local derivation
which is not a derivation on the algebra of all upper triangular complex $2\times2$ matrices.
S. Ayupov and K. Kudaybergenov \cite{Ayupov3} prove that each 2-local derivation on a von Neumann algebra is a derivation.
For more information about this topic, we refer to \cite{Ayupov2, He, Kim}.

In this paper, we study derivations, local derivations and 2-local derivations on some algebras of operators on Hilbert C*-modules.
There are few results in this topic.
P. Li, D. Han and W. Tang \cite{Lipengtong} prove that each derivation on End$^{*}_{\mathcal A}(\mathcal M)$ is inner,
where $\mathcal M$ is a full Hilbert C*-module over a commutative unital C*-algebra $\mathcal A$.
M. Moghadam, M. Miri and A. Janfada \cite{Mog} prove that each $\mathcal A$-linear derivation on End$_{\mathcal A}(\mathcal M)$ is inner,
where $\mathcal M$ is a full Hilbert C*-module over a commutative unital C*-algebra $\mathcal A$ with the property that there exist $x_0$ in $\mathcal M$
and $f_0$ in $\mathcal M^{'}$ such that $f_0(x_0)=e$.

Hilbert C*-modules provide a natural generalization of Hilbert spaces by replacing the complex field $\mathbb{C}$ with an arbitrary C*-algebra.
The theory of Hilbert C*-modules plays an important role in the theory of operator algebras, as it can be applied in many fields, such as index
theory of elliptic operators, K- and K K-theory, noncommutative geometry, and so on.

In the following, we would firstly review some properties of Hilbert C*-modules \cite{Lance}.

Let $\mathcal A$ be a C*-algebra and $\mathcal M$ be a left $\mathcal A$-module.

$\mathcal M$ is called a \emph{Pre-Hilbert~$\mathcal A$-module} if there exists a mapping
$\langle\cdot,\cdot\rangle:\mathcal M\times\mathcal M\longrightarrow\mathcal A$ with the following properties:
for each $\lambda\in\mathbb{C},a\in\mathcal A,x,y,z\in\mathcal M$,\\
(1) $\langle x,x\rangle\geq0$, and $\langle x,x\rangle=0$ implies that $x=0$,\\
(2) $\langle \lambda x+y,z\rangle=\lambda\langle x,z\rangle+\langle y,z\rangle$,\\
(3) $\langle ax,y\rangle=a\langle x,y\rangle$,\\
(4) $\langle x,y\rangle=\langle y,x\rangle^*$.

The mapping $\langle\cdot,\cdot\rangle$ is called an $\mathcal A$-valued inner product.
The inner product induces a norm on $\mathcal M$:~$\|x\|=\|\langle x,x\rangle\|^{1/2}$.
$\mathcal M$ is called a \emph{Hilbert~$\mathcal A$-module}(or more exactly, a Hilbert C*-module over $\mathcal A$),
 if it is complete with respect to this norm.

We denote by $\langle\mathcal M,\mathcal M\rangle$ the closure of the linear span of all the elements of the form $\langle x,y\rangle,x,y\in\mathcal M$.
$\mathcal M$ is called a \emph{full} Hilbert~$\mathcal A$-module
if $\langle\mathcal M,\mathcal M\rangle=\mathcal A$.

For a full Hilbert~$\mathcal A$-module $\mathcal M$,
we have the following lemma.
\begin{lemma}\label{l1}
Let $\mathcal A$ be a C*-algebra with unit $e$ and $\mathcal M$ be a full Hilbert~$\mathcal A$-module.
There exists a sequence $\{x_i\}_{i=1}^n\subseteq\mathcal M$,
such that $\sum_{i=1}^n\langle x_i,x_i\rangle=e$.
\end{lemma}

A linear mapping $T$ from $\mathcal M$ into itself is said to be \emph{$\mathcal A$-linear}
if $T(ax)=aT(x)$ for each $a\in\mathcal A$ and $x\in\mathcal M$.
A bounded $\mathcal A$-linear mapping from $\mathcal M$ into itself is called an \emph{operator} on $\mathcal M$.
Denote by End$_{\mathcal A}(\mathcal M)$ all operators on $\mathcal M$.
End$_{\mathcal A}(\mathcal M)$ is a Banach algebra.

A mapping $T$ from $\mathcal M$ into itself is said to be \emph{adjointable} if there exists a mapping $T^*$ such that
$\langle Tx,y\rangle=\langle x,T^*y\rangle$, for all $x,y\in\mathcal M$.
Notice that each adjointable mapping must be an operator. 
Denote by End$^{*}_{\mathcal A}(\mathcal M)$ all adjointable operators on $\mathcal M$.
End$^{*}_{\mathcal A}(\mathcal M)$ is a C*-algebra.

Similarly, a linear mapping $f$ from $\mathcal M$ into $\mathcal A$ is said to be $\mathcal A$-linear
if $f(ax)=af(x)$ for each $a\in\mathcal A$ and $x\in\mathcal M$.
The set of all bounded $\mathcal A$-linear mappings from $\mathcal M$ to $\mathcal A$ is denoted by $\mathcal M^{'}$.

For each $x$ in $\mathcal M$,
one can define a mapping $\hat{x}$ from $\mathcal M$ to $\mathcal A$
by follows: $\hat{x}(y)=\langle y,x\rangle$, for all $y\in\mathcal M$.
Obviously, $\hat{x}\in\mathcal M^{'}$.

For each $x$ in $\mathcal M$ and $f$ in $\mathcal M^{'}$,
one can define a mapping $\theta_{x,f}$ from $\mathcal M$ into itself by follows:
$\theta_{x,f}y=f(y)x$, for all $y\in\mathcal M$.
Obviously, $\theta_{x,f}\in$ End$_{\mathcal A}(\mathcal M)$.

In particular,
for each $x,y$ in $\mathcal M$,
we have $\theta_{x,\hat{y}}z=\hat{y}(z)x=\langle z,y\rangle x$, for all $z\in\mathcal M$.

For the operators of the above forms, we have the following lemmas.

\begin{lemma}\label{l2}
Let $\mathcal M$ be a Hilbert C*-module over a C*-algebra $\mathcal A$.\\
For all $a\in\mathcal A,x,y\in\mathcal M,f,g\in\mathcal M^{'},A\in$ End$_{\mathcal A}(\mathcal M)$, we have\\
(1) $\theta_{x,f}A=\theta_{x,f\circ A}$,\\
(2) $A\theta_{x,f}=\theta_{Ax,f}$,\\
(3) if in addition, $\mathcal A$ is commutative, then $\theta_{x,f}\theta_{y,g}=f(y)\theta_{x,g}$,~$\theta_{ax,f}=a\theta_{x,f}$.
\end{lemma}

\begin{lemma}\label{l3}
Let $\mathcal M$ be a Hilbert C*-module over a C*-algebra $\mathcal A$.\\
For all $a\in\mathcal A,x,y,z,w\in\mathcal M,A\in$ End*$_{\mathcal A}(\mathcal M)$, we have\\
(1) $\theta_{x,\hat{y}}\in$ End*$_{\mathcal A}(\mathcal M)$, and $\theta^*_{x,\hat{y}}=\theta_{y,\hat{x}}$,\\
(2) $\theta_{x,\hat{y}}A=\theta_{x,\hat{y}\circ A}=\theta_{x,\widehat{A^*y}}$,\\
(3) $A\theta_{x,\hat{y}}=\theta_{Ax,\hat{y}}$,\\
(4) if in addition, $\mathcal A$ is commutative, then $\theta_{x,\hat{y}}\theta_{z,\hat{w}}=\langle z,y\rangle\theta_{x,\hat{w}}$,~$\theta_{ax,\hat{y}}=a\theta_{x,\hat{y}}=\theta_{x,\widehat{a^*y}}$.
\end{lemma}

For a commutative C*-algebra $\mathcal A$,
for each $a$ in $\mathcal A$,
one can define a mapping $T_a$ from $\mathcal M$ into itself by follows:
$T_ax=ax$, for all $x\in\mathcal M$.
Obviously, $T_a\in$ End$_{\mathcal A}(\mathcal M)$.
It is worthwhile to notice that if $\mathcal A$ is not commutative,
then $T_a$ is not $\mathcal A$-linear.
In this case, $T_a$ is not in End$_{\mathcal A}(\mathcal M)$.

\begin{lemma}\label{l4}
Let $\mathcal A$ be a commutative C*-algebra with unit $e$ and $\mathcal M$ be a full Hilbert~$\mathcal A$-module.
Then $\mathcal Z($End$_{\mathcal A}(\mathcal M))=\{T_a:~a\in\mathcal A\}$.
\end{lemma}
\begin{proof}
For each $A$ in End$_{\mathcal A}(\mathcal M)$ and $x$ in $\mathcal M$,
since $AT_ax=A(ax)=aAx=T_aAx$, we have $AT_a=T_aA$.
It is to say $T_a\in\mathcal Z($End$_{\mathcal A}(\mathcal M))$.

On the other hand,
assume $A\in\mathcal Z($End$_{\mathcal A}(\mathcal M))$.
By Lemma~\ref{l1},
there exists a sequence $\{x_i\}_{i=1}^n\subseteq\mathcal M$,
such that $\sum_{i=1}^n\langle x_i,x_i\rangle=e$.
Thus we have
$$\sum_{i=1}^nA\theta_{x,\hat{x_i}}x_i=\sum_{i=1}^n\langle x_i,x_i\rangle Ax=Ax,$$
and
$$\sum_{i=1}^n\theta_{x,\hat{x_i}}Ax_i=\sum_{i=1}^n\langle Ax_i,x_i\rangle x.$$
Let $\sum_{i=1}^n\langle Ax_i,x_i\rangle=a$.
Then we have $A=T_a$.
The proof is complete.
\end{proof}

For an algebra $\mathcal A$,
if for each $a$ in $\mathcal A$,
$a\mathcal A a=0$ implies that $a=0$,
then it is said to be \emph{semi-prime}.

\begin{lemma}\label{l5}
Let $\mathcal A$ be a C*-algebra and $\mathcal M$ be a Hilbert~$\mathcal A$-module.
Then End$_{\mathcal A}(\mathcal M)$ is a semi-prime Banach algebra.
\end{lemma}
\begin{proof}
Let $A$ be in End$_{\mathcal A}(\mathcal M)$.
Assume that $ABA=0$ for each $B$ in End$_{\mathcal A}(\mathcal M)$.
In particular, for each $x\in \mathcal M$ and $f\in \mathcal M^{'}$,
we have
$$A\theta_{x,f}Ax=\theta_{Ax,f\circ A}x=f(Ax)Ax=0.$$
By taking $y=Ax$ and $f=\hat{y}$,
we have $\langle y,y\rangle y=0$.
It follows that
$$\langle\langle y,y\rangle y,\langle y,y\rangle y\rangle=\langle y,y\rangle^3=0.$$
Since $\langle y,y\rangle$ is a self-adjoint element,
we have $\langle y,y\rangle=0$,
and $y=0$. Hence $A=0$,
and End$_{\mathcal A}(\mathcal M)$ is semi-prime.
The proof is complete.
\end{proof}

\section{Derivations on End$_{\mathcal A}(\mathcal M)$}\

In this section, we study derivations on End$_{\mathcal A}(\mathcal M)$. We begin with several lemmas.

\begin{lemma}\label{l6}
Let $\mathcal A$ be a commutative unital C*-algebra and $\mathcal M$ be a full Hilbert~$\mathcal A$-module.
Then each derivation on $End_{\mathcal A}(\mathcal M)$ is $\mathcal A$-linear, i.e. $d(aA)=ad(A),$
for each $a\in \mathcal A$ and $A\in End_{\mathcal A}(\mathcal M)$.
\end{lemma}
\begin{proof}
Suppose $d$ is a derivation on End$_{\mathcal A}(\mathcal M)$.

By Lemma \ref{l4}, we have
$\mathcal Z($End$_{\mathcal A}(\mathcal M))=\{T_a:~a\in\mathcal A\}$.
For each $A$ in End$_{\mathcal A}(\mathcal M)$,
By
$$d(T_aA)=d(T_a)A+T_ad(A)$$
and
$$d(AT_a)=Ad(T_a)+d(A)T_a,$$
we obtain $d(T_a)A=Ad(T_a)$.
Hence $d(T_a)\in\mathcal Z($End$_{\mathcal A}(\mathcal M))$,
and $d(\mathcal Z($End$_{\mathcal A}(\mathcal M)))\subseteq\mathcal Z($End$_{\mathcal A}(\mathcal M))$.

Since $\mathcal Z($End$_{\mathcal A}(\mathcal M))=\{T_a:~a\in\mathcal A\}$ is a commutative C*-algebra,
and every derivation on a commutative C*-algebra is zero,
we have $d(T_a)=0$.

It follows that
$$d(aA)=d(T_aA)=d(T_a)A+T_ad(A)=T_ad(A)=ad(A),$$
which means that $d$ is $\mathcal A$-linear.
The proof is complete.
\end{proof}

\begin{lemma}\label{l7}
Let $\mathcal A$ be a commutative unital C*-algebra and $\mathcal M$ be a full Hilbert~$\mathcal A$-module.
Then each derivation on End$_{\mathcal A}(\mathcal M)$ is continuous.
\end{lemma}
\begin{proof}
Suppose $d$ is a derivation on End$_{\mathcal A}(\mathcal M)$.
Assume that $\{T_n\}$ is a sequence converging to zero in End$_{\mathcal A}(\mathcal M)$,
and $\{d(T_n)\}$ converges to $T$.

According to the closed graph theorem,
to show $d$ is continuous,
it is sufficient to prove that $T=0$.

By Lemma \ref{l6}, we know $d$ is $\mathcal A$-linear.
For $x,y\in \mathcal M$,~$f,g\in \mathcal M^{'}$, we have
$$d(\theta_{x,f}T_n\theta_{y,g})=d(f(T_ny)\theta_{x,g})=f(T_ny)d(\theta_{x,g})\rightarrow0,$$
and
$$d(\theta_{x,f}T_n\theta_{y,g})=d(\theta_{x,f})T_n\theta_{y,g}+\theta_{x,f}d(T_n)\theta_{y,g}+\theta_{x,f}T_nd(\theta_{y,g}).$$
Since $\{T_n\}$ converges to zero and $\{d(T_n)\}$ converges to $T$,
we have $$d(\theta_{x,f}T_n\theta_{y,g})\rightarrow\theta_{x,f}T\theta_{y,g}=f(Ty)\theta_{x,g}.$$
It follows that $f(Ty)\theta_{x,g}=0.$

Let $a=f(Ty)$,
then we have $a\theta_{x,g}=\theta_{ax,g}=0$.
For each $z\in \mathcal M$,
we have 
\begin{align}
\theta_{ax,g}z=g(z)ax=0.\label{e1}
\end{align}
By taking $g=\widehat{ax}$ and $z=ax$ in \eqref{e1},
we can obtain $ax=0$,
i.e. 
\begin{align}
f(Ty)x=0.\label{e2}
\end{align}
By taking $f=\widehat{Ty}$ and $x=Ty$ in \eqref{e2},
we can obtain $Ty=0$.
i.e. $T=0$.
The proof is complete.
\end{proof}

Now we can prove our main theorem in this section.

\begin{theorem}\label{t1}
Let $\mathcal A$ be a commutative C*-algebra with unit $e$ and $\mathcal M$ be a full Hilbert~$\mathcal A$-module.
Then each derivation on End$_{\mathcal A}(\mathcal M)$ is an inner derivation.
\end{theorem}
\begin{proof}
Suppose $d$ is a derivation on End$_{\mathcal A}(\mathcal M)$ and
$\{x_i\}_{i=1}^n$ is a sequence in $\mathcal M$ such that $\sum_{i=1}^n\langle x_i,x_i\rangle=e$.

Define a mapping $T$ from $\mathcal M$ into itself by follows:
$$Tx=\sum_{i=1}^nd(\theta_{x,x_i})x_i,$$ 
for all $x\in \mathcal M.$

By Lemmas \ref{l6} and \ref{l7},
$d$ is $\mathcal A$-linear and continuous,
thus $T$ is also $\mathcal A$-linear and continuous.
That is to say $T\in $ End$_{\mathcal A}(\mathcal M)$.

Now it is sufficient to show that $d(A)=TA-AT$, for each $A\in $ End$_{\mathcal A}(\mathcal M)$.

For each $x\in \mathcal M$,
we have
\begin{align*}
TAx&=\sum_{i=1}^nd(\theta_{Ax,x_i})x_i\\
&=\sum_{i=1}^nd(A\theta_{x,x_i})x_i\\
&=\sum_{i=1}^nd(A)\theta_{x,x_i}x_i+\sum_{i=1}^nAd(\theta_{x,x_i})x_i\\
&=d(A)\sum_{i=1}^n\langle x_i,x_i\rangle x+A\sum_{i=1}^nd(\theta_{x,x_i})x_i\\
&=d(A)x+ATx.
\end{align*}

It implies that $d(A)=TA-AT$.
Hence $d$ is an inner derivation.
The proof is complete.
\end{proof}

\section{2-Local derivations on End$_{\mathcal A}(\mathcal M)$ and End$^*_{\mathcal A}(\mathcal M)$}\

In this section, we characterize 2-local derivations on End$_{\mathcal A}(\mathcal M)$ and End$^*_{\mathcal A}(\mathcal M)$.
Firstly, we show the following lemma.

\begin{lemma}\label{l8}
Let $\mathcal A$ be a commutative unital C*-algebra and $\mathcal M$ be a Hilbert~$\mathcal A$-module.
For $x_i\in\mathcal M$ and $f_i\in \mathcal M^{'}$,
if~ $\sum_{i=1}^n\theta_{x_i,f_i}=0$,
then $\sum_{i=1}^nf_i(x_i)=0$.
\end{lemma}
\begin{proof}
Let $a_{i,j}=f_j(x_i)\in\mathcal A$
and $\Lambda=(a_{i,j})_{n\times n}\in M_n(\mathcal A)$.

We have
\begin{align*}
&\sum_{i=1}^nf_i(x_k)x_i=\sum_{i=1}^n\theta_{x_i,f_i}x_k=0\\
\Rightarrow &\sum_{i=1}^nf_i(x_k)f_j(x_i)=f_j(\sum_{i=1}^nf_i(x_k)x_i)=0\\
\Rightarrow &\sum_{i=1}^na_{k,i}a_{i,j}=0\\
\Rightarrow &\Lambda^2=0.
\end{align*}

Since $\mathcal A$ is a commutative unital C*-algebra,
it is well known that $\mathcal A$ is $*$-isomorphic to $C(\mathcal S)$ for some compact Hausdorff space $\mathcal S$.
Without loss of generality, we can assume $\mathcal A=C(\mathcal S)$.

Then for each $t\in \mathcal S$,
we have $a_{i,j}(t)\in\mathbb{C}$ and $\Lambda(t),\Lambda^2(t) \in M_n(\mathbb{C})$.

Recall that for a matrix $A$ in $M_n(\mathbb{C})$,
$A^2=0$ implies that $tr(A)=0$,
where $tr(A)$ denotes the trace of $A$,
i.e. the sum of all the diagonal elements.

Hence $\Lambda^2(t)=0$ implies that $tr(\Lambda(t))=0$.
It follows that $tr(\Lambda)=0$,
that is to say $\sum_{i=1}^nf_i(x_i)=0$.
The proof is complete.
\end{proof}

\begin{theorem}\label{t2}
Let $\mathcal A$ be a commutative unital C*-algebra and $\mathcal M$ be a full Hilbert~$\mathcal A$-module.
Then each 2-local derivation on End$_{\mathcal A}(\mathcal M)$ is a derivation.
\end{theorem}
\begin{proof}
Denote by $\Gamma(\mathcal M)$ the linear span of the set $\{\theta_{x,f}:x\in\mathcal M,f\in\mathcal M^{'}\}$.
By Lemma \ref{l2}, $\Gamma(\mathcal M)$ is a two-side ideal of End$_{\mathcal A}(\mathcal M)$.

For each $S=\sum_{i=1}^n\theta_{x_i,f_i}\in\Gamma(\mathcal M)$,
define $\phi(S)=\sum_{i=1}^nf_i(x_i)$.

One can verify that $\phi$ is well defined by Lemma \ref{l8}.
And obviously, $\phi$ is $\mathcal A$-linear.
Moreover, for each $A\in $ End$_{\mathcal A}(\mathcal M)$, we have
$$\phi(\theta_{x,f}A)=\phi(\theta_{x,f\circ A})=f(Ax)=\phi(\theta_{Ax,f})=\phi(A\theta_{x,f}).$$

It follows that $\phi(SA)=\phi(AS)$ for each $A\in$ End$_{\mathcal A}(\mathcal M)$ and $S\in\Gamma(\mathcal M)$.

Suppose $\delta$ is a 2-local derivation on End$_{\mathcal A}(\mathcal M)$.
By the definition of 2-local derivation,
there exists a derivation $d$ on End$_{\mathcal A}(\mathcal M)$ such that
$\delta(A)=d(A)$ and $\delta(S)=d(S)$.
By Theorem \ref{t1},
$d$ is an inner derivation,
i.e. there exists an element $T\in$ End$_{\mathcal A}(\mathcal M)$ such that $d=D_T$.

Thus we have
$$\delta(A)S+A\delta(S)=d(A)S+Ad(S)=d(AS)=D_T(AS)=TAS-AST.$$

Since $\Gamma(\mathcal M)$ is a two-side ideal of End$_{\mathcal A}(\mathcal M)$,
we know that $AS\in\Gamma(\mathcal M)$.

Hence
$$\phi(\delta(A)S+A\delta(S))=\phi(TAS-AST)=0,$$
which follows that $\phi(\delta(A)S)=-\phi(A\delta(S))$.

Now, for each $A,B\in$ End$_{\mathcal A}(\mathcal M)$ and $S\in\Gamma(\mathcal M)$,
we have
\begin{align*}
\phi(\delta(A+B)S)&=-\phi((A+B)\delta(S))\\
&=-\phi(A\delta(S))-\phi(B\delta(S))\\
&=\phi(\delta(A)S)+\phi(\delta(B)S)\\
&=\phi((\delta(A)+\delta(B))S).
\end{align*}

Let $C=\delta(A+B)-\delta(A)-\delta(B)$,
we obtain $\phi(CS)=0$.

By taking $S=\theta_{x,f}$,
we have
$$\phi(C\theta_{x,f})=f(Cx)=0\Rightarrow\langle Cx,Cx\rangle=0\Rightarrow Cx=0\Rightarrow C=0.$$

It means that $\delta(A+B)=\delta(A)+\delta(B)$.
That is to say $\delta$ is an additive mapping.
In addition, by the definition of 2-local derivation, it is easy to show that $\delta$ is homogeneous and $\delta(A^2)=A\delta(A)+\delta(A)A$
for each $A\in$ End$_{\mathcal A}(\mathcal M)$.
Hence $\delta$ is a Jordan derivation.

By Lemma \ref{l5},
End$_{\mathcal A}(\mathcal M)$ is a semi-prime Banach algebra.
According to the classical result that every Jordan derivation on a semi-prime Banach algebra
is a derivation \cite{Cusack}, we obtain that $\delta$ is a derivation.
The proof is complete.
\end{proof}

\begin{theorem}\label{t3}
Let $\mathcal A$ be a commutative unital C*-algebra and $\mathcal M$ be a full Hilbert~$\mathcal A$-module.
Then each 2-local derivation on End*$_{\mathcal A}(\mathcal M)$ is a derivation.
\end{theorem}
\begin{proof}
Denote by $\Gamma^*(\mathcal M)$ the linear span of the set $\{\theta_{x,\hat{y}}:x,y\in\mathcal M\}$.
By Lemma \ref{l3}, $\Gamma^*(\mathcal M)$ is a two-side ideal of End$^*_{\mathcal A}(\mathcal M)$.

For each $S=\sum_{i=1}^n\theta_{x_i,\widehat{y_i}}\in\Gamma^*(\mathcal M)$,
define $\phi(S)=\sum_{i=1}^n\langle x_i,y_i\rangle$.

By Lemma \ref{l8}, $\phi$ is well defined.
For each $A\in $ End$^*_{\mathcal A}(\mathcal M)$, we have
$$\phi(\theta_{x,\widehat{y}}A)=\phi(\theta_{x,\widehat{A^*y}})=\langle x,A^*y\rangle=\langle Ax,y\rangle=\phi(\theta_{Ax,\widehat{y}})=\phi(A\theta_{x,\widehat{y}}).$$

It follows that $\phi(SA)=\phi(AS)$ for each $A\in$ End$^*_{\mathcal A}(\mathcal M)$ and $S\in\Gamma^*(\mathcal M)$.

In \cite{Lipengtong}, the authors prove that for a commutative unital C*-algebra $\mathcal A$ and a full Hilbert~$\mathcal A$-module $\mathcal M$,
each derivation on End$^*_{\mathcal A}(\mathcal M)$ is an inner derivation.

The rest of the proof is similar to Theorem \ref{t2}, so we omit it.
\end{proof}

\section{Local derivations on End$_{\mathcal A}(\mathcal M)$}\

In this section, we discuss local derivations on End$_{\mathcal A}(\mathcal M)$.
Through this section, we assume that $\mathcal A$ is a commutative C*-algebra with unit $e$,
and $\mathcal M$ is a Hilbert~$\mathcal A$-module, and moreover, there exist $x_0$ in $\mathcal M$
and $f_0$ in $\mathcal M^{'}$ such that $f_0(x_0)=e$. Denote the unit of End$_{\mathcal A}(\mathcal M)$ by $I$.
Define
$\mathcal{L}=span\{\theta_{x,f_0}:x\in\mathcal M\}$, and $\mathcal{R}=span\{\theta_{x_0,f}:f\in\mathcal M^{'}\}.$

\begin{lemma}\label{l9}
~\\
(1) $\theta_{x_0,f_0}$ is an idempotent;\\
(2) each element in $\mathcal{L}$ is an $\mathcal A$-linear combination of some idempotents in $\mathcal{L}$,
and each element in $\mathcal{R}$ is an $\mathcal A$-linear combination of some idempotents in $\mathcal{R}$;\\
(3) $\mathcal{L}$ is a left ideal of $End_{\mathcal A}(\mathcal M)$,
and $\mathcal{R}$ is a right ideal of $End_{\mathcal A}(\mathcal M)$;\\
(4) $\mathcal{L}$ is a left separating set of End$_{\mathcal A}(\mathcal M)$, 
 i.e. for each $A$ in End$_{\mathcal A}(\mathcal M)$, $A\mathcal{L}=0$ implies that $A=0$,
and $\mathcal{R}$ is a right separating set of End$_{\mathcal A}(\mathcal M)$,
i.e. for each $A$ in End$_{\mathcal A}(\mathcal M)$, $\mathcal{R}A=0$ implies that $A=0$
\end{lemma}
\begin{proof}
(1) $\theta_{x_0,f_0}\theta_{x_0,f_0}=f_0(x_0)\theta_{x_0,f_0}=\theta_{x_0,f_0}$.\\
(2) For each $x\in\mathcal M$,
there exists a non-zero complex number $\lambda\in\mathbb{C}$,
such that $e-\lambda f_0(x)$ is invertible in $\mathcal A$.
Denote $e-\lambda f_0(x)$ by $a^{-1}$, then we have
$$f_0(a(x_0-\lambda x))=af_0(x_0-\lambda x)=a(e-\lambda f_0(x))=aa^{-1}=e.$$

By (1), we know that $\theta_{a(x_0-\lambda x),f_0}$ is an idempotent.

Thus we have
$$\theta_{x,f_0}=\lambda^{-1}\theta_{x_0,f_0}-\lambda^{-1}a^{-1}\theta_{a(x_0-\lambda x),f_0}.$$

That is to say $\theta_{x,f_0}$ is an $\mathcal A$-linear combination of idempotents in $\mathcal{L}$.

Similarly, for each $f\in\mathcal M^{'}$,
there exists a non-zero complex number $\lambda\in\mathbb{C}$,
such that $e-\lambda f(x_0)$ is invertible in $\mathcal A$.
Denote $e-\lambda f(x_0)$ by $a^{-1}$, then we have
$$(a(f_0-\lambda f))(x_0)=a(e-\lambda f(x_0))=aa^{-1}=e.$$

Again by (1), we know that $\theta_{x_0,a(f_0-\lambda f)}$ is an idempotent.

Thus we have
$$\theta_{x_0,f}=\lambda^{-1}\theta_{x_0,f_0}-\lambda^{-1}a^{-1}\theta_{x_0,a(f_0-\lambda f)}.$$\\
(3) For each $A\in$ End$_{\mathcal A}(\mathcal M),$
since $A\theta_{x,f_0}=\theta_{Ax,f_0}$, we know that $\mathcal{L}$ is a left ideal of End$_{\mathcal A}(\mathcal M)$.
Similarly,
$\mathcal{R}$ is a right ideal of End$_{\mathcal A}(\mathcal M)$
since $\theta_{x_0,f}A=\theta_{x_0,f\circ A}$.\\
(4) Suppose $A\in$ End$_{\mathcal A}(\mathcal M),\theta_{x,f_0}\in\mathcal L,\theta_{x_0,f}\in\mathcal R$.

If $A\theta_{x,f_0}=0$,
then
$$0=A\theta_{x,f_0}x_0=\theta_{Ax,f_0}x_0=f_0(x_0)Ax=Ax,$$
i.e. $A=0$.

If $\theta_{x_0,f}A=0$,
then
\begin{align*}
&\theta_{x_0,f}Ax=f(Ax)x_0=0\\
\Rightarrow &f_0(f(Ax)x_0)=f(Ax)f_0(x_0)=f(Ax)=0\\
\Rightarrow &\langle Ax,Ax\rangle=0\\
\Rightarrow &Ax=0\\
\Rightarrow &A=0.
\end{align*}

The proof is complete.
\end{proof}

Let $\mathcal J$ be a left $\mathcal A$-module,
and $\phi$ be a bilinear mapping from End$_{\mathcal A}(\mathcal M)\times$ End$_{\mathcal A}(\mathcal M)$ into $\mathcal J$.

We say that $\phi$ is \emph{$\mathcal A$-bilinear} if $\phi(aA,B)=\phi(A,aB)=a\phi(A,B)$ for each $A,B\in$ End$_{\mathcal A}(\mathcal M)$ and $a\in\mathcal A$.

We say that $\phi$ preserves zero product if $AB=0$ implies that $\phi(A,B)=0$ for each $A,B\in$ End$_{\mathcal A}(\mathcal M)$.

\begin{lemma}\label{l10}
Let $\mathcal J$ be a left $\mathcal A$-module,
and $\phi: End_{\mathcal A}(\mathcal M)\times End_{\mathcal A}(\mathcal M) \rightarrow\mathcal J$ be an $\mathcal A$-bilinear mapping preserving zero product.
Then for each $A,B\in End_{\mathcal A}(\mathcal M),L\in\mathcal L$, and $R\in\mathcal R$,
we have:
\begin{align}
\phi(A,LB)=\phi(AL,B)=\phi(I,ALB)
\end{align}
and
\begin{align}
\phi(AR,B)=\phi(A,RB)=\phi(ARB,I).
\end{align}
\end{lemma}
\begin{proof}
Suppose $P$ is an idempotent in End$_{\mathcal A}(\mathcal M)$.
Let $Q=I-P$.

Since $\phi$ preserves zero product,
we have
$$\phi(A,PB)=\phi(AP+AQ,PB)=\phi(AP,PB)=\phi(AP,B-QB)=\phi(AP,B).$$

By Lemma \ref{l9}(2),
each element in $\mathcal{L}$ is an $\mathcal A$-linear combination of idempotents in $\mathcal{L}$.
Considering $\phi$ is $\mathcal A$-bilinear,
we obtain that $\phi(A,LB)=\phi(AL,B)$.

By Lemma \ref{l9}(3),
$\mathcal L$ is a left ideal,
so $AL\in\mathcal L$.
Hence $\phi(AL,B)=\phi(I,ALB)$.

Similarly, we can show the equation (4.2) is true.
\end{proof}

For an algebra $\mathcal A$ with unit $e$, a linear mapping $\delta$ on $\mathcal A$ is said to be a \emph{generalized derivation}
if $\delta(ab)=a\delta(b)+\delta(a)b-a\delta(e)b$, for all $a, b$ in $\mathcal A$.
\begin{theorem}\label{t4}
Suppose that $\mathcal A$ is a commutative C*-algebra with unit $e$,
and $\mathcal M$ is a Hilbert~$\mathcal A$-module, and moreover, there exist $x_0$ in $\mathcal M$
and $f_0$ in $\mathcal M^{'}$ such that $f_0(x_0)=e$.
If $\delta$ is an $\mathcal A$-linear mapping from $End_{\mathcal A}(\mathcal M)$ into itself such that:
for each $A,B,C$ in End$_{\mathcal A}(\mathcal M)$, $AB=BC=0$ implies that $A\delta(B)C=0$,
then $\delta$ is a generalized derivation.
In particular,
if $\delta(I)=0$,
where $I$ is the unit of End$_{\mathcal A}(\mathcal M)$,
then $\delta$ is a derivation.
\end{theorem}
\begin{proof}
Suppose $A,B,X,Y,A_0,B_0$ are arbitrary elements in End$_{\mathcal A}(\mathcal M)$,
where $A_0B_0=0$,
$L$ and $R$ are arbitrary elements in $\mathcal L$ and $\mathcal R$, respectively.

Define a bilinear mapping $\phi_1$:
$\phi_1(X,Y)=X\delta(YA_0)B_0$.
Then $\phi_1$ is an $\mathcal A$-bilinear mapping preserving zero product.

By Lemma \ref{l10},
we have
$$\phi_1(R,A)=\phi_1(RA,I),$$
i.e.
$$R\delta(AA_0)B_0=RA\delta(A_0)B_0.$$

Since $\mathcal{R}$ is a right separating set of End$_{\mathcal A}(\mathcal M)$,
we have
$$\delta(AA_0)B_0=A\delta(A_0)B_0.$$

Now define a bilinear mapping $\phi_2$:
$\phi_2(X,Y)=\delta(AX)Y-A\delta(X)Y$.
Then $\phi_2$ is also an $\mathcal A$-bilinear mapping preserving zero product.

Again by Lemma \ref{l10},
we have
$$\phi_2(B,L)=\phi_2(I,BL),$$

i.e.
$$\delta(AB)L-A\delta(B)L=\delta(A)BL-A\delta(I)BL.$$

Since $\mathcal{L}$ is a left separating set of End$_{\mathcal A}(\mathcal M)$,
we obtain that
$$\delta(AB)=A\delta(B)+\delta(A)B-A\delta(I)B.$$

That is to say $\delta$ is a generalized derivation.
The proof is complete.
\end{proof}

Applying the above Theorem, we can get the following corollary immediately.

\begin{corollary}
Suppose $\mathcal A$ is a commutative C*-algebra with unit $e$,
$\mathcal M$ is a Hilbert~$\mathcal A$-module, and moreover, there exist $x_0$ in $\mathcal M$
and $f_0$ in $\mathcal M^{'}$ such that $f_0(x_0)=e$.
Then each $\mathcal A$-linear local derivation $\delta$ on $End_{\mathcal A}(\mathcal M)$ is a derivation.
\end{corollary}
\begin{proof}
For each $A,B,C$ in End$_{\mathcal A}(\mathcal M)$,
if $AB=BC=0$,
by the definition of local derivation,
there exists a derivation $\delta_B$ such that $\delta_B(B)=\delta(B)$.
Thus we have
$$A\delta(B)C=A\delta_B(B)C=\delta_B(ABC)-\delta_B(A)BC-AB\delta_B(C)=0.$$

Let $I$ be the unit of End$_{\mathcal A}(\mathcal M)$,
by the definition of local derivation,
there exists a derivation $\delta_I$ such that $\delta_I(I)=\delta(I)=0$.

By Theorem \ref{t4}, $\delta$ is a derivation.
The proof is complete.
\end{proof}

\emph{Acknowledgements}. This paper was partially supported by National Natural Science Foundation of China(Grant No. 11371136).

\end{document}